\documentclass[12pt]{article}
\usepackage{amsmath,amsthm, amssymb, amsfonts,latexsym,amscd}
\usepackage[dvips]{graphicx}
\newtheorem{theorem}{Theorem}[section]
\newtheorem{corollary}[theorem]{Corollary}
\newtheorem{lemma}[theorem]{Lemma}
\newtheorem{proposition}[theorem]{Proposition}

\newtheorem{example}[theorem]{Example}

\begin{document}
\title{\bf Minimal cut-sets in the power graphs of certain finite non-cyclic groups}
\author{Sriparna Chattopadhyay\footnote{Supported by SERB NPDF scheme (File No. PDF/2017/000908), Department of Science and Technology, Government of India} \hskip 0.5cm Kamal Lochan Patra \hskip 0.5cm Binod Kumar Sahoo}
\maketitle

\begin{abstract}
The power graph of a group is the simple graph with vertices as the group elements, in which two distinct vertices are adjacent if and only if one of them can be obtained as an integral power of the other. We study (minimal) cut-sets of the power graph of a (finite) non-cyclic (nilpotent) group which are associated with its maximal cyclic subgroups. Let $G$ be a finite non-cyclic nilpotent group whose order is divisible by at least two distinct primes. If $G$ has a Sylow subgroup which is neither cyclic nor a generalized quaternion $2$-group and all other Sylow subgroups of $G$ are cyclic, then under some conditions
we prove that there is only one minimum cut-set of the power graph of $G$. We apply this result to find the vertex connectivity of the power graphs of certain finite non-cyclic abelian groups whose order is divisible by at most three distinct primes. \\

\end{abstract}

\section{Introduction}

Let $\Gamma$ be a simple graph with vertex set $V$. A subset $X$ of $V$ is called a (vertex) {\it cut-set} of $\Gamma$ if the induced subgraph of $\Gamma$ with vertex set $V\setminus X$ is  disconnected. So $|V\setminus X|\geq 2$ for any cut-set $X$ of $\Gamma$. A cut-set $X$ is called a {\it minimal cut-set} if $X\setminus \{x\}$ is not a cut-set of $\Gamma$ for any $x\in X$. If $X$ is a minimal cut-set of $\Gamma$, then any proper subset of $X$ is not a cut-set of $\Gamma$. A cut-set $X$ of $\Gamma$ is called a {\it minimum cut-set} if $|X|\leq |Y|$ for any cut-set $Y$ of $\Gamma$. Clearly, every minimum cut-set of $\Gamma$ is also a minimal cut-set. The {\it vertex connectivity} of $\Gamma$, denoted by $\kappa(\Gamma)$, is the minimum number of vertices which need to be removed from $V$ so that the induced subgraph of $\Gamma$ on the remaining vertices is disconnected or has only one vertex. The latter case arises only when $\Gamma$ is a complete graph. If $\Gamma$ is not a complete graph and $X$ is a minimum cut-set of $\Gamma$, then $\kappa(\Gamma)=|X|$. A {\it separation} of $\Gamma$ is a pair $(A,B)$, where $A,B$ are disjoint non-empty subsets of $V$ whose union is $V$ and there is no edge of $\Gamma$ containing vertices from both $A$ and $B$. Thus, $\Gamma$ is disconnected if and only if there exists a separation of it. We refer to \cite{BM} for the unexplained terminology from graph theory used in this paper.

\subsection{Power graph}

The notion of directed power graph of a group was introduced in \cite{ker1}, which was further extended to semigroups in \cite{ker1-1, ker2}. Then the undirected power graph of a semigroup, in particular, of a group was defined in \cite{ivy}. Several researchers have studied both the directed and undirected power graphs of groups from different viewpoints. More on these graphs can be found in the survey paper \cite{survey} and the references therein.

The {\it power graph} $\mathcal{P}(G)$ of a group $G$ is the simple (undirected) graph with vertex set $G$, in which two distinct vertices are adjacent if and only if one of them can be obtained as an integral power of the other. Thus two distinct vertices $x,y\in G$ are adjacent if and only if $x\in\langle y\rangle$ or $y\in\langle x\rangle$. By definition, every non-identity element $x$ of $G$ is adjacent to $x^0=1$ (the identity element of $G$) and so $\mathcal{P}(G)$ is always connected.

The power graph of a finite group is complete if and only if the group is cyclic of prime power order \cite[Theorem 2.12]{ivy}. It was proved in \cite[Theorem 1.3]{cur} and \cite[Corollary 3.4]{cur-1} that, among all finite groups of a given order, the cyclic group of that order has the maximum number of edges and has the largest clique in its power graph.

For a subset $A$ of $G$, we denote by $\mathcal{P}(A)$ the induced subgraph of $\mathcal{P}(G)$ with vertex set $A$. The subgraph $\mathcal{P}^{\ast}(G)=\mathcal{P}(G\setminus\{1\})$ of $\mathcal{P}(G)$ is called the {\it proper power graph} of $G$. In \cite{mog}, the authors proved connectedness of the proper power graph of certain groups. For the dihedral group $D_{2n}$ of order $2n$, the identity element is a cut-vertex of $\mathcal{P}(D_{2n})$ and so $\mathcal{P}^{\ast}(D_{2n})$ is disconnected. If $G$ is one of the groups $PGL(2,p^n)$ ($p$ an odd prime), $PSL(2,p^n)$ ($p$ prime), or a Suzuki group $Sz(2^{2n+1})$, then $\mathcal{P}^{\ast}(G)$ is disconnected \cite[Theorems 3.5--3.7]{doos}. In \cite[Section 4]{doos}, the authors proved that $\mathcal{P}^{\ast}(S_n)$ and $\mathcal{P}^{\ast}(A_n)$ are disconnected for many values of $n$, where $S_n, A_n$ are the symmetric and alternating groups respectively.

\subsection{Vertex connectivity}

For any given group, determining the vertex connectivity of its power graph is an interesting problem. Clearly, any cut-set of the power graph contains the identity element of the group. So the vertex connectivity of the power graph is $1$ if and only if the group is of order $2$ or its proper power graph is disconnected. We recall a few results on the vertex connectivity of the power graph of finite $p$-groups and cyclic groups. If $G$ is a cyclic $p$-group, then $\mathcal{P}(G)$ is a complete graph and so $\kappa\left(\mathcal{P}\left(G\right)\right)=|G|-1$. If $G$ is a dicyclic group (in particular, a generalized quaternion $2$-group), then the set consisting of the identity element and the unique involution of $G$ is a minimum cut-set of $\mathcal{P}(G)$ and so $\kappa\left(\mathcal{P}\left(G\right)\right)=2$ \cite[Theorem 7]{sri}. If $G$ is a finite $p$-group, then $\mathcal{P}^{\ast}(G)$ is connected if and only if $G$ is either cyclic or a generalized quaternion $2$-group by \cite[Corollary 4.1]{mog} (also see \cite[Theorem 2.6 (1)]{doos}). In particular, $\kappa(\mathcal{P}(G))=1$ if $G$ is a finite non-cyclic abelian $p$-group, in this case the number of connected components of $\mathcal{P}^{\ast}(G)$ is obtained in \cite[Theorem 3.3]{panda}.

Let $C_n$ be the finite cyclic group of order $n$. The number of generators of $C_n$ is $\phi(n)$, where $\phi$ is the Euler's totient function. Recall that $\phi$ is a multiplicative function, that is, $\phi(ab)=\phi(a)\phi(b)$ for any two positive integers $a,b$ which are relatively prime. Also, $\phi(p^k)=p^{k-1}(p-1)$ for any prime $p$ and positive integer $k$. The identity element and the generators of $C_n$ are adjacent to all other vertices of $\mathcal{P}(C_n)$. So any cut-set of $\mathcal{P}(C_n)$ must contain these elements, giving
$\kappa(\mathcal{P}(C_n))\geq \phi(n) +1$. Further, equality holds if and only if $n$ is a prime or a product of two distinct primes, see \cite[Lemma 2.5]{cps}.

For the rest of the paper, we take $r\geq 1$, $n_1,n_2,\ldots, n_r$ are positive integers and $p_1,p_2,\ldots,p_r$ are prime numbers with $p_1<p_2<\cdots <p_r$. From Theorems 1.3, 1.5 and Corollary 1.4 of \cite{cps}, we have the following.

\begin{theorem}\cite{cps}\label{cyclic-connectivity}
Suppose that $r\geq 2$. Then the following hold:
\begin{enumerate}
\item[(i)]  If $n=p_1^{n_1}p_2^{n_2}\ldots p_r^{n_r}$ and $2\phi(p_1p_2\ldots p_{r-1}) > p_1p_2\ldots p_{r-1}$, then
$$\kappa(\mathcal{P}(C_n))=\phi(n) + p_1^{n_1 -1}\ldots p_{r-1}^{n_{r-1} -1} p_r^{n_r -1} \left[p_1p_2\ldots p_{r-1} - \phi(p_1p_2\ldots p_{r-1})\right].$$
\item[(ii)] If $n=p_1^{n_1} p_2^{n_2}$, then $\kappa(\mathcal{P}(C_n))=\phi \left(p_{1}^{n_{1}}p_{2}^{n_2}\right)+p_{1}^{n_{1}-1}p_{2}^{n_2-1}$.
\item[(iii)] If $n=p_1^{n_1} p_2^{n_2}p_3^{n_3}$, then the following hold.
\begin{enumerate}
\item[(a)] $p_1=2:$ $\kappa(\mathcal{P}(C_n))=\phi(n)+2^{n_{1}-1}p_{2}^{n_{2}-1}\left[(p_2 -1)p_{3}^{n_{3}-1}+2\right]$.
\item[(b)] $p_1\geq 3:$ $\kappa(\mathcal{P}(C_n))= \phi(n) + p_1^{n_1 -1} p_{2}^{n_{2} -1} p_3^{n_3 -1} \left[p_1+p_2 - 1\right]$.
\end{enumerate}
\end{enumerate}
In each of the above cases, there is only one minimum cut-set of $\mathcal{P}(C_n)$ except when $(r,p_1)=(2,2)$. If $(r,p_1)=(2,2)$, then there are $n_2$ minimum cut-sets of $\mathcal{P}(C_n)$.
\end{theorem}

We note that Theorem \ref{cyclic-connectivity}(ii) was also proved in \cite[Theorem 2.38]{panda} and Theorem \ref{cyclic-connectivity}(iii) for $n_1=n_2=n_3=1$ was also obtained in \cite[Theorem 2.40]{panda}. Recently, the authors obtained in \cite{cps-1} the value of $\kappa(\mathcal{P}(C_n))$ for the cases: (i) $n_r\geq 2$ and (ii) $n$ is a product of distinct primes.

\subsection{Main results}

In Section \ref{sec-noncyclic}, we shall study (minimal) cut-sets of the power graph of a (finite) non-cyclic (nilpotent/abelian) group which are associated with its maximal cyclic subgroups. We then prove the following results in Sections \ref{sec-main-1}, \ref{sec-main-2}, \ref{sec-main-3} respectively.

\begin{theorem}\label{main-1}
Let $G$ be a finite non-cyclic nilpotent group of order $p_1^{n_1}p_2^{n_2}\ldots p_r^{n_r}$, $r\geq 2$. For $1\leq i \leq r,$ let $P_i$ be the Sylow $p_i$-subgroup of $G$. Suppose that each Sylow subgroup is cyclic except $P_{k}$ for some $k\in\{1,2,\ldots, r\}$ and that $P_k$ is not a generalized quaternion $2$-group if $(k,p_1)=(1,2)$. Set $Q=P_{1}\cdots P_{k-1}P_{k+1}\cdots P_{r}$. If $p_k\geq r+1$ or if $2\phi(p_1\ldots p_{r-1}) > p_1\ldots p_{r-1}$, then $Q$ is the only minimum cut-set of $\mathcal{P}(G)$ and hence
$\kappa(\mathcal{P}(G))=n/p_{k}^{n_{k}}$.
\end{theorem}

For a cyclic group $H$, we denote by $\widetilde{H}$ the set of all non-generators of $H$.

\begin{theorem}\label{main-2}
Let $G$ be a finite non-cyclic abelian group of order $p_1^{n_1}p_2^{n_2}$ and $P_i$ be the Sylow $p_i$-subgroup of $G$, $i=1,2$. Then the following hold.
\begin{enumerate}
\item[(i)] Suppose that either $P_1$ or $P_2$ is non-cyclic. If $P_i$ is non-cyclic, then $P_j$ is a minimum cut-set of $\mathcal{P}(G)$ and so $\kappa(\mathcal{P}(G))=|P_j|=p_j^{n_j}$, where $\{i,j\}=\{1,2\}$. In fact, if $p_1\geq 3$, or if $p_1=2$ and $P_2$ is non-cyclic, then there is only one minimum cut-set of $\mathcal{P}(G)$.
\item[(ii)] Suppose that both $P_1$ and $P_2$ are non-cyclic. If $p_1\geq 3$ and $G$ has maximal cyclic subgroup $C$ of order $p_1p_2$, then $\kappa(\mathcal{P}(G))=\min \left\{|P_1|,|P_2|,|\widetilde{C}|\right\}$.
\item[(iii)] Suppose that both $P_1$ and $P_2$ are non-cyclic and that $P_1$ is elementary abelian. Then $\kappa(\mathcal{P}(G))=\min \left\{|P_1|,|P_2|,|\widetilde{C}|\right\}$, where $C$ is a maximal cyclic subgroup of $G$ of minimum possible order.
\end{enumerate}
\end{theorem}

We note that if $p_1=2$, $P_1$ is non-cyclic and $P_2$ is cyclic in Theorem \ref{main-2}(i), then there might be more than one minimum cut-set of $\mathcal{P}(G)$, see Example \ref{example}.

\begin{theorem}\label{main-3}
Let $G$ be a finite non-cyclic abelian group of order $p_1^{n_1}p_2^{n_2}p_3^{n_3}$ and $P_i$ be the Sylow $p_i$-subgroup of $G$ for $i\in\{1,2,3\}$. Suppose that exactly one Sylow subgroup of $G$ is non-cyclic. Then the following hold.
\begin{enumerate}
\item[(i)] If $p_1=2$ and $P_1$ is non-cyclic, then $\kappa(\mathcal{P}(G))=\min \left\{|P_2P_{3}|,\kappa(\mathcal{P}(C))\right\}$, where $C$ is a maximal cyclic subgroup of $G$ of minimum possible order. More precisely, if $|C|=2^{c}p_2^{n_2}p_3^{n_3}$ for some positive integer $c$, then
\begin{equation*}
 \kappa(\mathcal{P}(G))=
\begin{cases}
|P_2P_3|, &\text{if $c>1$,}\\
\kappa(\mathcal{P}(C)), &\text{if $c=1$.}
\end{cases}
\end{equation*}
\item[(ii)] If  $p_1=2$ and $P_k$ is non-cyclic, then $P_1P_j$ is the only minimum cut-set of $\mathcal{P}(G)$ and so $\kappa(\mathcal{P}(G))=|P_1P_j|=p_1^{n_1}p_j^{n_j}$, where $\{j,k\}=\{2, 3\}$.
\item[(iii)] If $p_1\geq 3$ and $P_k$ is non-cyclic, then $P_iP_j$ is the only minimum cut-set of $\mathcal{P}(G)$ and so $\kappa(\mathcal{P}(G))=|P_iP_j|=p_i^{n_i}p_j^{n_j}$, where $\{i,j,k\}=\{1,2, 3\}$.
\end{enumerate}
\end{theorem}

Note that, the value of $\kappa(\mathcal{P}(C))$ in Theorem \ref{main-3}(i) can be obtained using the formula in Theorem \ref{cyclic-connectivity}(iii).

\section{Non-cyclic groups and minimal cut-sets}\label{sec-noncyclic}

Let $G$ be any group, need not be finite. For distinct elements $x,y\in G$, we write $x \sim y$ if they are adjacent in $\mathcal{P}(G)$. For $w\in G$, let $[w]$ be the set of all generators of the cyclic subgroup $\langle w \rangle$. Then $\mathcal{P}([w])$ is a clique in $\mathcal{P}(G)$. The following fundamental result is a generalization of \cite[Lemma 2.3]{cps} which can be obtained from \cite[Theorem 2.16]{panda}.

\begin{lemma}\label{equi-class}
If $X$ is a minimal cut-set of $\mathcal{P}(G)$, then either $[w]\subseteq X$ or $[w]\cap X=\emptyset$ for every $w\in G$.
\end{lemma}

For $x\in G$, let $N(x)$ be the neighborhood of $x$ in $\mathcal{P}(G)$, that is, the set of all elements of $G$ which are adjacent to $x$. Note that either $N(x)=G\setminus \{x\}$ or $N(x)$ is a cut-set of $\mathcal{P}(G)$. In the latter case, $N(x)$ is not a minimal cut-set if the order of $x$ is at least $3$. This can be seen by taking $X=N(x)$ and $w=x$ in Lemma \ref{equi-class}, also see \cite[Theorem 2.29]{panda}. \medskip

A cyclic subgroup $M$ of $G$ is called a {\it maximal cyclic subgroup} if it is not properly contained in any cyclic subgroup of $G$. So if $G$ is a cyclic group, then the only maximal cyclic subgroup of $G$ is itself. In a finite group, every element is contained in a maximal cyclic subgroup. However, this statement need not hold in an infinite group: $(\mathbb{Q},+)$ is an example.

We denote by $\mathcal{M}(G)$ the collection of all maximal cyclic subgroups of $G$. For $M\in\mathcal{M}(G)$, the following result ensures that $\widetilde{M}$ is always a cut-set of $\mathcal{P}(G)$ when $G$ is non-cyclic.

\begin{proposition}\label{all-group}
Suppose that $G$ is a non-cyclic group and let $M\in\mathcal{M}(G)$. If $A:=G\setminus M$ and $B:=M\setminus \widetilde{M}$, then $(A,B)$ is a separation of $\mathcal{P}(G\setminus \widetilde{M})$. In particular,  $\widetilde{M}$ is a cut-set of $\mathcal{P}(G)$.
\end{proposition}

\begin{proof}
Clearly, $A\cap B$ is empty and $A\cup B= G\setminus \widetilde{M}$. Since $G$ is not cyclic, $M$ is properly contained in $G$ and so $A$ is non-empty. Also, $B$ is non-empty as it consists of all the generators of $M$.

Let $b\in B$. For any integer $k$, the element $b^k$ is either in $B$ or in $\widetilde{M}$ according as it is a generator of $M$ or not. This implies that no element of $A$ can be obtained as a power of any element of $B$.
Conversely, suppose that $a^t\in B$ for some $a\in A$ and some integer $t$. Then $a\notin M$. By the definition of the set $B$, $a^t$ is a generator of $M$ and so $M=\langle a^t\rangle\leq \langle a\rangle$. Since $M$ is a maximal cyclic subgroup of $G$, we must have $M= \langle a^t\rangle= \langle a\rangle$. This gives $a\in M$, a contradiction. So no element of $B$ can be obtained as a power of any element of $A$.

Thus there is no edge in $\mathcal{P}(G\setminus \widetilde{M})$ containing one element from $A$ and the other from $B$. Therefore, $(A,B)$ is a separation of $\mathcal{P}(G\setminus \widetilde{M})$ and hence $\widetilde{M}$ is a cut-set of $\mathcal{P}(G)$.
\end{proof}

\begin{corollary}
If $G$ is a finite non-cyclic group, then $\kappa(\mathcal{P}(G))\leq \min\{|\widetilde{M}|:M\in\mathcal{M}(G)\}$.
\end{corollary}

For $M\in\mathcal{M}(G)$, let $\overline{M}$ be the union of all sets of the form $M\cap \langle y\rangle$, where $y\in G\setminus M$. Observe that the generators of any member of $\mathcal{M}(G)$ (in particular, of $M$) are not contained in $\overline{M}$. So $\overline{M}\subseteq \widetilde{M}\subsetneq M$. There are groups for which $\overline{M}$ is properly contained in $\widetilde{M}$ for every $M\in\mathcal{M}(G)$ (see Proposition \ref{M-M-1}).

\begin{proposition}\label{M-bar}
If $G$ is a non-cyclic group, then $\overline{M}$ is a cut-set of $\mathcal{P}(G)$ for every $M\in\mathcal{M}(G)$.
\end{proposition}

\begin{proof}
Let $M=\langle x\rangle$, $x\in G$. Since $G$ is non-cyclic, there exists $y\in G\setminus M$. Thus $y\notin \langle x\rangle$ and $x,y\notin \overline{M}$. Since $M$ is a maximal cyclic subgroup, we also have $x\notin \langle y\rangle$.

Suppose that $\overline{M}$ is not a cut set of $\mathcal{P}(G)$. Then there exists a path from $x$ to $y$ in $\mathcal{P}(G\setminus\overline{M})$. Let $x=x_0,x_1,x_2,\cdots, x_n=y$ be such a path, where $x_k\notin \overline{M}$ for $k\in\{0,1,\ldots,n\}$. We have $n\geq 2$ as $x\notin \langle y\rangle$ and $y\notin \langle x\rangle$. Let $j\in\{1,2,\cdots, n\}$ be the smallest integer such that $x_j\notin \langle x\rangle$. We then have $x_{j-1}\in\langle x_j\rangle$ as $x_{j-1}\sim x_j$. Thus $x_{j-1}$ is an element of $\langle x\rangle\cap \langle x_j\rangle=M\cap \langle x_j\rangle$, where $x_j\notin \langle x\rangle=M$. This gives $x_{j-1}\in\overline{M}$, contradicting that each $x_k$ is outside $\overline{M}$.
\end{proof}

The proof of Proposition \ref{M-bar} is similar to that of \cite[Theorem 10]{mir}, in which the following corollary was obtained.

\begin{corollary}
If $G$ is a finite non-cyclic group, then $\kappa(\mathcal{P}(G))\leq \min\{|\overline{M}|:M\in\mathcal{M}(G)\}$.
\end{corollary}

We shall prove in the next section that $\overline{M}$ is a minimal cut-set of $\mathcal{P}(G)$ if $G$ is a finite nilpotent group with at least two non-cyclic Sylow subgroups (see Proposition \ref{M-bar-min}).

\begin{proposition}\label{gen-connected}
Suppose that $G$ is a non-cyclic group and $X$ is a minimal cut-set of $\mathcal{P}(G)$. Then the following two statements are equivalent:
\begin{enumerate}
\item[(1)] $X$ has no element which will generate a member of $\mathcal{M}(G)$.
\item[(2)] $\mathcal{P}(M\setminus X)$ is connected for every $M\in\mathcal{M}(G)$.
\end{enumerate}
\end{proposition}

\begin{proof}
Clearly, $(1)\Rightarrow (2)$, since all the generators of $M$ are in $M\setminus X$. Now, assume (2). On the contrary, suppose that $X$ contains a generator, say $\alpha$, of some $M\in\mathcal{M}(G)$. Set $\widehat{X}=X\setminus\{\alpha\}$ and fix a separation $(A,B)$ of $\mathcal{P}(G\setminus X)$.
Since $X$ is a minimal cut-set, the subgraph $\mathcal{P}(G\setminus \widehat{X})$ is connected. So there exists $a\in A$ and $b\in B$ such that $a\sim \alpha\sim b$. Since $\alpha$ generates a maximal cyclic subgroup of $G$, each of $a$ and $b$ can be obtained as a power of $\alpha$. It follows that $a,b$ are two distinct elements in $M\setminus X$.
By our assumption, $\mathcal{P}(M\setminus X)$ is connected. This means that there is a path between $a$ and $b$ in $\mathcal{P}(M\setminus X)$ and hence in $\mathcal{P}(G\setminus X)$, contradicting that $(A,B)$ is a separation of $\mathcal{P}(G\setminus X)$.
\end{proof}

\begin{proposition}\label{existence-0}
Suppose that $G$ is a non-cyclic group in which every element is contained in a maximal cyclic subgroup of $G$. Let $X$ be a cut-set of $\mathcal{P}(G)$. If there exists a subset $D$ of $G$ which is contained in every member of $\mathcal{M}(G)$ but not contained in $X$, then there exists $M\in \mathcal{M}(G)$ such that $\mathcal{P}(M\setminus X)$ is disconnected.
\end{proposition}

\begin{proof}
Suppose that $\mathcal{P}(M\setminus X)$ is connected for every $M\in \mathcal{M}(G)$. Consider two distinct elements $x_1,x_2$ of $G\setminus X$ and let $M_1,M_2\in \mathcal{M}(G)$ containing $x_1,x_2$, respectively. Since $D\nsubseteq X$, there exists an element $u\in D\setminus X$. Since each of $M_1$ and $M_2$ contains $D$, we have $u\in M_1\setminus X$ and $u\in M_2\setminus X$. By our assumption, each of $\mathcal{P}(M_1\setminus X)$ and $\mathcal{P}(M_2\setminus X)$ is connected. So either $u=x_j$ or there is a path between $u$ and $x_j$ in $\mathcal{P}(M_j\setminus X)$, $j=1,2$. This implies that there is a path between $x_1$ and $x_2$ in $\mathcal{P}(G\setminus X)$. Since $x_1,x_2\in G\setminus X$ are arbitrary, it follows that $\mathcal{P}(G\setminus X)$ is connected, a contradiction.
\end{proof}

\subsection{Finite non-cyclic nilpotent groups}

For $x\in G$, we shall denote by $o(x)$ the order of $x$. The following two lemmas are useful, the first one can be found in \cite[5.3.6]{robinson} and the second one in \cite[Theorem 10.9]{rose}.

\begin{lemma}\label{p-group}\cite{robinson}
Let $G$ be a finite $p$-group, $p$ prime. Then $G$ has exactly one subgroup of order $p$ if and only if it is cyclic, or $p=2$ and $G$ is a generalized quaternion $2$-group.
\end{lemma}

\begin{lemma} \label{nilpotent-equi}\cite{rose}
Let $G$ be a finite nilpotent group. Then the following hold:
\begin{enumerate}
\item[(i)] If $a,b\in G$ and $o(a), o(b)$ are relatively prime, then $a$ and $b$ commute in $G$.
\item[(ii)] All Sylow subgroups of $G$ are normal in $G$.
\item[(iii)] $G$ is isomorphic to the direct product of its Sylow subgroups.
\end{enumerate}
\end{lemma}

We shall use Lemma \ref{nilpotent-equi}(i) mostly without mention. The following elementary result is used frequently while defining paths/walks between to distinct connected vertices.

\begin{lemma}\label{xypath}
Let $G$ be a finite nilpotent group. If $x,y$ are two non-identity elements of $G$ such that $o(x)$ and $o(y)$ are relatively prime, then $x\sim xy\sim y$ in $\mathcal{P}(G)$.
\end{lemma}

\begin{proof}
Since $o(x)$ and $o(y)$ are relatively prime, we have $\left\langle x^{o(y)} \right\rangle=\langle x \rangle$ and so $\left(x^{o(y)}\right)^{t}=x$ for some integer $t$. Since $xy=yx$, $(xy)^{o(y) t}=x^{o(y) t}=x$ and so $x\sim xy$. A similar argument proves that $y\sim xy$.
\end{proof}

Now consider $G$ to be a {\it finite non-cyclic nilpotent group} with $|G|=p_1^{n_1}p_2^{n_2}\ldots p_r^{n_r}$.
For $1\leq i \leq r,$ let $P_i$ be the Sylow $p_i$-subgroup of $G$. Then at least one $P_i$ is non-cyclic. So $r\geq 2$ if some $P_j$ is cyclic. We have $|P_i|=p_i^{n_i}$ and $G=P_1P_2\cdots P_r$, an internal direct product of $P_1,P_2,\ldots,P_r$. If $M_i$ is a cyclic subgroup of $P_i$ for $1\leq i\leq r$, then $M_1M_2\cdots M_r$ is a cyclic subgroup of $G$.

Conversely, let $M$ be a cyclic subgroup of $G$. Then $M=\langle a \rangle$ for some $a\in G$. As $G=P_1P_2\cdots P_r$, we can write $a=a_1a_2\ldots a_r$, where $a_i\in P_i$ for $1\leq i \leq r$. Let $o(a_i)=p_i^{k_i}$ for some $k_i$ with $k_i\leq n_i$. Then $o(a)=p_1^{k_1}p_2^{k_2}\ldots p_r^{k_r}$. We have
$$a^{p_1^{k_1}\ldots p_{i-1}^{k_{i-1}}p_{i+1}^{k_{i+1}}\ldots p_r^{k_r}}=a_i^{p_1^{k_1}\ldots p_{i-1}^{k_{i-1}}p_{i+1}^{k_{i+1}}\ldots p_r^{k_r}}.$$
Since $p_i^{k_i}$ and $p_1^{k_1}\ldots p_{i-1}^{k_{i-1}}p_{i+1}^{k_{i+1}}\ldots p_r^{k_r}$ are relatively prime, we get
$$\left\langle a_i^{p_1^{k_1}\ldots p_{i-1}^{k_{i-1}}p_{i+1}^{k_{i+1}}\ldots p_r^{k_r}} \right\rangle=\langle a_i \rangle.$$
Set $M_i=\langle a_i \rangle$, a cyclic subgroup of $P_i$. Then it follows that $M_i$ is a subgroup of $M$ for each $i$ and hence $M_1M_2\cdots M_r$ is a subgroup of $M$. As $|M_i|=o(a_i)=p_i^{k_i}$, we have $|M_1M_2\cdots M_r|=|M_1||M_2|\cdots |M_r|=p_1^{k_1}p_2^{k_2}\ldots p_r^{k_r}=o(a)=|M|$ and so $M=M_1M_2\cdots M_r$. From the above discussion, it follows that $M$ is a maximal cyclic subgroup of $G$ if and only if $M_i$ is a maximal cyclic subgroup of $P_i$ for each $i$. Thus we have the following.

\begin{lemma}\label{max-cyclic}
Any maximal cyclic subgroup of $G$ is of the form $M_1M_2\cdots M_r$, where $M_i$ is a maximal cyclic subgroup of $P_i$, $1\leq i\leq r$.
\end{lemma}

\begin{proposition}\label{size-compare}
Let $C\in\mathcal{M}(G)$ be of minimum possible order. Then $|\widetilde{M}|\geq |\widetilde{C}|$ for every $M\in\mathcal{M}(G)$.
\end{proposition}

\begin{proof}
Let $M\in\mathcal{M}(G)$. By Lemma \ref{max-cyclic}, $M=M_1M_2\cdots M_r$ and $C=C_1C_2\cdots C_r$ for some $M_i, C_i\in \mathcal{M}(P_i)$, $1\leq i\leq r$. Since $C\in\mathcal{M}(G)$ is of minimum order, $C_i$ must be a maximal cyclic subgroup of $P_i$ of minimum order and so $|C_i|\leq |M_i|$ for each $i$. Let $|M_i|=p_i^{k_i}$ and $|C_i|=p_i^{s_i}$. Then $1\leq s_{i}\leq k_{i}$ for $i\in\{1,2,\ldots, r\}$.
Since $|M|=p_{1}^{k_{1}}p_{2}^{k_{2}}\ldots p_r^{k_r}$ and $|C|=p_{1}^{s_{1}}p_{2}^{s_{2}}\ldots p_r^{s_r}$, we get
\begin{eqnarray*}
|\widetilde{M}| & = & p_{1}^{k_{1}}p_{2}^{k_{2}}\ldots p_r^{k_r} - \phi\left(p_{1}^{k_{1}}p_{2}^{k_{2}}\ldots p_r^{k_r}\right)\\
& = & p_{1}^{k_{1}-1}p_{2}^{k_{2}-1}\ldots p_r^{k_r-1}[p_{1}p_{2}\ldots p_r-\phi(p_{1}p_{2}\ldots p_r)]\\
& \geq & p_{1}^{s_{1}-1}p_{2}^{s_{2}-1}\ldots p_r^{s_r-1}[p_{1}p_{2}\ldots p_r - \phi(p_{1}p_{2}\ldots p_r)]\\
& = & p_{1}^{s_{1}}p_{2}^{s_{2}}\ldots p_r^{s_r} - \phi\left(p_{1}^{s_{1}}p_{2}^{s_{2}}\ldots p_r^{s_r}\right)=|\widetilde{C}|.
\end{eqnarray*}
In the above, the first and last equality holds since the number of non-generators in a cyclic group of order $m$ is $m-\phi(m)$.
\end{proof}

\begin{proposition}\label{existence-1}
Suppose that $r\geq 2$ and $P_j$ is cyclic for some $j\in\{1,2,\ldots ,r\}$. If $X$ is a cut-set of $\mathcal{P}(G)$ not containing $P_j$, then $X$ contains an element which generates a maximal cyclic subgroup of $G$.
\end{proposition}

\begin{proof}
Since $P_j$ is cyclic, Lemma \ref{max-cyclic} implies that every maximal cyclic subgroup of $G$ contains $P_j$. Taking $D=P_j$ in Proposition \ref{existence-0} and using the given hypothesis that $P_j$ is not contained in $X$, we get that there exists $M\in \mathcal{M}(G)$ such that $\mathcal{P}(M\setminus X)$ is disconnected. Then the result follows from Proposition \ref{gen-connected}.
\end{proof}

\begin{proposition}\label{contain}
Suppose that $r\geq 2$. Let $X$ be a cut-set of $\mathcal{P}(G)$ not containing any Sylow subgroup of $G$ and $(A,B)$ be a separation of $\mathcal{P}(G\setminus X)$. If $\mathcal{P}(M\setminus X)$ is connected for every $M\in\mathcal{M}(G)$, then the sets $P_i\setminus X$ and $P_j\setminus X$ are contained either in $A$ or in $B$, $1\leq i\neq j\leq r$.
\end{proposition}

\begin{proof}
Since $P_i\setminus X\neq \emptyset$ and $P_j\setminus X\neq \emptyset$, there exist $u_i\in P_i\setminus X$ and $u_j\in P_j\setminus X$. Without loss, we may assume that $u_i\in A$. We show that both $P_i\setminus X$ and $P_j\setminus X$ are contained in $A$.

Let $y\in P_j\setminus X$. Let $M$ be a maximal cyclic subgroup of $G$ containing both $u_i$ and $y$. Since $\mathcal{P}(M\setminus X)$ is connected by our assumption, there is a path between $u_i$ and $y$ in $\mathcal{P}(M\setminus X)$ and hence in $\mathcal{P}(G\setminus X)$. Since $u_i\in A$ and $(A,B)$ is a separation of $\mathcal{P}(G\setminus X)$, it follows that $y$ must be in $A$. So $P_j\setminus X$ is contained in $A$. In particular, taking $y=u_j$, we get $u_j\in A$. Then applying a similar argument as above, we get that $P_i\setminus X$ is contained in $A$.
\end{proof}

\begin{proposition}\label{all-group1}
Suppose that $G$ has at least two Sylow subgroups which are non-cyclic. Then $\mathcal{P}(G\setminus M)$ is connected for every $M\in\mathcal{M}(G)$. In particular, $G\setminus M$ and $M\setminus \widetilde{M}$ are the only two connected components of $\mathcal{P}(G\setminus \widetilde{M})$.
\end{proposition}

\begin{proof}
Set $A=G\setminus M$ and $B=M\setminus \widetilde{M}$. Then $(A,B)$ is a separation of $\mathcal{P}(G\setminus \widetilde{M})$ by Proposition \ref{all-group}. Clearly, $\mathcal{P}(B)$ is connected, in fact, a clique. We show that $\mathcal{P}(A)$ is connected.

Let the Sylow subgroups $P_{j}$ and $P_{k}$ be non-cyclic for some $j$ and $k$ with $1\leq j<k\leq r$. Take $S=P_1\cdots P_{j-1}P_{j+1}\cdots P_{k}\cdots P_{r}$. Then $P_{j}$ and $S$ are non-cyclic subgroups of $G$ and $G=P_{j}S$. So both the sets $\overline{P_j}:=P_j\setminus (M\cap P_j)$ and $\overline{S}:=S\setminus (M\cap S)$ are non-empty and contained in $A$.
Note that if $x\in \overline{P_j}$ and $y\in \overline{S}$, then $xy\notin M$. This can be seen as follows. Since $o(x)$ and $o(y)$ are relatively prime, we have $\left\langle (xy)^{o(y)}\right\rangle= \left\langle x^{o(y)}\right\rangle=\left\langle x\right\rangle$. If $xy\in M$, then $\left\langle (xy)^{o(y)}\right\rangle$ is a subgroup of $M$ and so also $\left\langle x\right\rangle$. This gives $x\in M$, contradicting that $x\in \overline{P_j}$.

Now, take two arbitrary distinct elements $a_js$ and $x_js'$ of $A$, where $a_j,x_j\in P_j$ and $s,s'\in S$. We show that there is a walk between $a_js$ and $x_js'$ in $\mathcal{P}(A)$. Since $a_js\notin M$, we have $a_j\in \overline{P_j}$ or $s\in \overline{S}$. Similarly, $x_j\in \overline{P_j}$ or $s' \in \overline{S}$. We shall use Lemma \ref{xypath} frequently for adjacencies in the following walks. If both $s,s'$ are not in $\overline{S}$, then $a_j,x_j$ are in $\overline{P_j}$. Taking any $\overline{s}\in \overline{S}$, we get a walk
$$a_js\sim a_j\sim a_j \overline{s}\sim \overline{s}\sim x_j\overline{s}\sim x_j\sim x_js'$$
between $a_js$ and $x_js'$ in $\mathcal{P}(A)$. If both $a_j,x_j$ are not in $\overline{P_j}$, then both $s,s'$ are in $\overline{S}$. Take any $u_j\in \overline{P_j}$ to produce a walk
$$a_js\sim s\sim u_j s\sim u_j\sim u_js'\sim s'\sim x_js'$$
between $a_js$ and $x_js'$ in $\mathcal{P}(A)$. For each of the other possible cases, one of the following
$$a_js\sim s\sim x_j s\sim x_j\sim x_js'\text{\;\;\;\; or \;\;\;\;}a_js\sim a_j\sim a_j s'\sim s'\sim x_js'$$
will define a walk between $a_js$ and $x_js'$ in $\mathcal{P}(A)$. This completes the proof.
\end{proof}

\begin{proposition}\label{M-bar-min}
If $G$ has at least two Sylow subgroups which are non-cyclic, then $\overline{M}$ is a minimal cut-set of $\mathcal{P}(G)$ for every $M\in\mathcal{M}(G)$.
\end{proposition}

\begin{proof}
By Proposition \ref{M-bar}, $\overline{M}$ is a cut-set of $\mathcal{P}(G)$. Let $C=G\setminus M$ and $D=M\setminus\overline{M}$. Then $\mathcal{P}(D)$ is connected as $D$ contains the generators of $M$. Since $G$ has at least two non-cyclic Sylow subgroups, $\mathcal{P}(C)$ is connected by Proposition \ref{all-group1}. It follows that $(C,D)$ is the only separation of $\mathcal{P}(G\setminus\overline{M})$.

In order to prove the minimality of $\overline{M}$, we need to show that $\mathcal{P}(G\setminus (\overline{M}-\{\alpha\}))$ is connected for every $\alpha\in \overline{M}$. Clearly, $\alpha$ is adjacent with every element of $D$ which are generators of $M$. By the definition of $\overline{M}$, there exists $\beta\in C=G\setminus M$ such that $\alpha\in\langle\beta\rangle$, that is, $\beta\sim\alpha$. Since both $\mathcal{P}(C)$ and $\mathcal{P}(D)$ are connected, it follows that $\mathcal{P}(G\setminus (\overline{M}-\{\alpha\}))$ is connected.
\end{proof}

\begin{proposition}\label{nilpotent-mincut}
Suppose that $r\geq 2$ and $P_k$ is neither cyclic nor a generalized quaternion $2$-group for some $k\in\{1,2,\ldots ,r\}$. Then $Q=P_{1}\cdots P_{k-1}P_{k+1}\cdots P_r$ is a  minimal cut-set of $\mathcal{P}(G)$.
\end{proposition}

\begin{proof}
By Lemma \ref{p-group}, $P_k$ has at least two distinct subgroups each of prime order $p_{k}$. Let $x,y\in P_{k}$ be such that $o(x)=p_{k}=o(y)$ and $\langle x \rangle \neq \langle y \rangle$. Then $x,y\in G\setminus Q$ and they are non-adjacent in $\mathcal{P}(G)$.

We first prove that $Q$ is a cut-set of $\mathcal{P}(G)$, by showing that there is no path between $x$ and $y$ in $\mathcal{P}(G\setminus Q)$. On the contrary, suppose that there is path between $x$ and $y$ in $\mathcal{P}(G\setminus Q)$. Let
$$x \sim z_{1} \sim z_{2} \sim \cdots \sim z_{m-1}\sim z_m=y$$
be a shortest path of length $m$ between $x$ and $y$ in $\mathcal{P}(G\setminus Q)$. Clearly, $m\geq 2$ as $x$ and $y$ are not adjacent in $\mathcal{P}(G)$.
Consider the first three elements $x,z_1,z_2$ in the above path. All of them are outside $Q$ and $x$ is not adjacent with $z_2$. If $z_1\in\langle x\rangle$, then $\langle x\rangle=\langle z_1\rangle$ as $o(x)=p_k$ is a prime. Since $z_1\sim z_2$, it would then follow that $x\sim z_2$, a contradiction. So $x\in \langle z_1\rangle$. If $z_1\in \langle z_2\rangle$, then $x\in \langle z_2\rangle$ and so $x\sim z_2$, again a contradiction. Hence $z_2\in \langle z_1\rangle$.

If $p_k | o(z_2)$, then $\langle z_{2} \rangle$ has a unique subgroup $H$ of order $p_{k}$. Since $\langle z_{2} \rangle \leq \langle z_{1} \rangle$, $H$ is a subgroup of $\langle z_{1} \rangle$. Since $\langle z_{1} \rangle$ has a unique subgroup of order $p_k$ and $\langle x\rangle$ is already such a subgroup of $\langle z_{1} \rangle$, we must have $\langle x \rangle=H$. Then it follows that $x\in \langle z_2 \rangle$ and so $x\sim z_2$, a contradiction. So $p_k\nmid o(z_2)$ and hence $z_2\in Q$, a final contradiction. Thus $\mathcal{P}(G\setminus Q)$ is disconnected.

We now prove the minimality of the cut-set $Q$. Let $w\in Q$ be arbitrary and set $\widehat{Q}=Q\setminus\{w\}$. We show that $\mathcal{P}(G\setminus \widehat{Q})$ is connected. This is clear if $w=1$. So assume that $w\neq 1$. Let $\alpha,\beta\in G\setminus \widehat{Q}$ with $\alpha\neq \beta$. As $G=P_kQ,$ we can write $\alpha=u_ku$ and $\beta=v_kv$ for some $u_k,v_k\in P_k$ and $u,v\in Q$. Then, by Lemma \ref{xypath},
$$\alpha=u_ku\sim u_k\sim u_k w\sim w\sim v_k w\sim v_k\sim v_kv=\beta$$
is a walk between $\alpha$ and $\beta$ in $\mathcal{P}(G\setminus \widehat{Q})$. Since $\alpha,\beta$ are arbitrary elements of $G\setminus \widehat{Q}$, it follows that $\mathcal{P}(G\setminus \widehat{Q})$ is connected.
\end{proof}

\subsection{Finite non-cyclic abelian groups}

We need the following elementary result.

\begin{lemma}\label{equiva}
Let $p,q$ be distinct prime numbers and $r,m$ be positive integers. Then there exists an integer $l$ such that $pl\equiv m\mod q^r$.
\end{lemma}

\begin{proof}
Since $p$ and $q^r$ are relatively prime, there exist integers $s$ and $t$ such that $ps+q^r t=1$. Then $psm+q^r tm=m$. Take $l=sm$. Then $pl-m=q^r (-tm)$ and so $pl\equiv m\mod q^r$.
\end{proof}

\begin{proposition}\label{nongen-adjacency}
Let $G$ be a finite non-cyclic abelian group and $M$ be a maximal cyclic subgroup of $G$. Then the following two statements are equivalent:
\begin{enumerate}
\item[(1)] For every $\alpha\in \widetilde{M}$, there exists  $\beta\in G\setminus M$ such that $\alpha\in\langle\beta\rangle$.
\item[(2)] Every Sylow subgroup of $G$ is non-cyclic.
\end{enumerate}
\end{proposition}

\begin{proof}
Since $G$ is non-cyclic, $G\setminus M$ is non-empty and at least one Sylow subgroup is non-cyclic. By Lemma \ref{max-cyclic}, $M=M_1M_2\cdots M_r$ for some $M_i\in\mathcal{M}(P_i)$, $1\leq i\leq r$. Let $M_i=\langle w_i\rangle$ for some $w_i\in P_i$ and $|M_i|=o(w_i)=p_i^{r_i}$ for some $1\leq r_i\leq n_i$. Note that any generator of $M$ is of the form $w_1^{l_1}w_2^{l_2} \ldots w_r^{l_r}$, where $l_i$ and $p_i^{r_i}$ are relatively prime.

$(1) \Rightarrow(2)$: Without loss, we may assume that $P_1,\ldots,P_k$ are non-cyclic and $P_{k+1},\ldots,P_r$ are cyclic for some $k\in\{1,2,\ldots, r\}$. We show that $k=r$. On the contrary, suppose that $k< r$ (so $r\geq 2$). Then $x=w_1w_2\cdots w_k$ is a non-generator of $M$. By (1), there exists $y\in G\setminus M$ such that $y^t=x$ for some integer $t$. Since $G=P_1P_2\cdots P_r$, we can write $y=y_1y_2\cdots y_r$ for some $y_i\in P_i$, $1\leq i\leq r$. Then $y_1^t y_2^t\cdots y_k^t y_{k+1}^t\cdots y_r^t =y^t=x=w_1w_2\cdots w_k$. Since $G=P_1P_2\cdots P_r$ is a direct product, it follows that $y_1^t=w_1$, $y_2^t=w_2,\cdots ,y_k^t=w_k$ and $y_i^t=1$ for $k+1\leq i\leq r$. Thus $w_j\in\langle y_j\rangle$ for $1\leq j\leq k$. Since $M_j=\langle w_j\rangle$ is a maximal cyclic subgroup of $P_j$, we must have $M_j=\langle w_j\rangle =\langle y_j\rangle$ and so $y_j\in M_j$ for $1\leq j\leq k$. Now, for $k+1\leq i\leq r$, $P_i$ is cyclic implies $M_i=P_i$ and so $y_i\in M_i$. Thus $y=y_1y_2\cdots y_r\in M_1M_2\cdots M_r=M$, contradicting that $y\in G\setminus M$. So $k=r$ and hence each Sylow subgroup of $G$ is non-cyclic.

$(2) \Rightarrow(1)$: Since $P_i$ is not cyclic, the set $\overline{P_i}=P_i\setminus M_i$ is non-empty, and by Lemma \ref{p-group}, there exists a subgroup of $P_i$ of order $p_i$ which is different from the unique subgroup of order $p_i$ contained in $M_i$. So there exists $u_i\in \overline{P_i}$ of order $p_i$. Since $M_i=M\cap P_i$, indeed $u_i\in G\setminus M$ for each $i\in\{1,2,\ldots,r\}$.

Let $\alpha\in\widetilde{M}$. If $\alpha=1$, then any element of $G\setminus M$ can be taken as $\beta$. Assume that $\alpha\neq 1$. Write $\alpha=w_1^{m_1}w_2^{m_2}\ldots w_r^{m_r}$ for some integers $m_1,m_2, \ldots, m_r$ with $0\leq m_i<p_i^{r_i}$. Since $\alpha\neq 1$, at least one $m_i$ is non-zero. Suppose that $m_j= 0$ for some $j$. Then $p_j$ and $o(\alpha)$ are relatively prime. So $\left\langle \alpha^{p_j} \right\rangle=\langle \alpha \rangle$ and hence $\alpha^{p_j t}=\alpha$ for some integer $t$. Take $\beta=u_j\alpha$. Then $\beta\in G\setminus M$ as $u_j\notin M$ and $\beta^{p_j t}=(u_j\alpha)^{p_j t}=\alpha^{p_j t}=\alpha$. So assume that all $m_i$'s are non-zero.

Since $\alpha$ is not a generator of $M$, we have $p_i\mid m_i$ for at least one $i$. Without loss of generality, suppose that $p_1\mid m_1$.
By Lemma \ref{equiva}, there exist integers $k_2,\ldots, k_r$ such that $p_1 k_j\equiv m_j\mod p_j^{r_j}$ for $j=2,\ldots, r$.
Now take $$\beta=u_1 w_1^{\frac{m_1}{p_1}}w_2^{k_2}\ldots w_r^{k_r}.$$
Then $\beta\notin M$ as $u_1\notin M$, that is, $\beta\in G\setminus M$. Calculating $\beta^{p_1}$, we get
$$\beta^{p_1}=\left(u_1 w_1^{\frac{m_1}{p_1}}w_2^{k_2}\ldots w_r^{k_r}\right)^{p_1} = w_1^{m_1}w_2^{p_1 k_2}\ldots w_r^{p_1k_r}=w_1^{m_1}w_2^{m_2}\ldots w_r^{m_r}=\alpha.$$
Here the second equality holds as $o(u_1)=p_1$, and the third equality holds using the facts that $o(w_j)=p_j^{r_j}$ and $p_1 k_j\equiv m_j\mod p_j^{r_j}$ for $j=2,\ldots,r$. Thus $\alpha\in\langle\beta\rangle$.
\end{proof}

As a consequence of Proposition \ref{nongen-adjacency}, we have the following.

\begin{corollary}\label{coro-non-adj}
Let $G$ be a finite non-cyclic abelian group in which every Sylow subgroup is non-cyclic. For $\alpha\in G$, if $\langle \alpha\rangle\notin\mathcal{M}(G)$, then $\alpha$ is contained in at least two distinct members of $\mathcal{M}(G)$.
\end{corollary}

We prove the following result relating $\overline{M}$ and $\widetilde{M}$ for $M\in\mathcal{M}(G)$.

\begin{proposition}\label{M-M-1}
Let $G$ be a finite non-cyclic abelian group. For $M\in\mathcal{M}(G)$, $\overline{M}=\widetilde{M}$ if and only if every Sylow subgroup of $G$ is non-cyclic.
\end{proposition}

\begin{proof}
Proposition \ref{nongen-adjacency} implies that $\widetilde{M}\subseteq \overline{M}$ if and only if every Sylow subgroup of $G$ is non-cyclic. Then the result follows from the fact that $\overline{M}\subseteq \widetilde{M}$ always. \end{proof}

\begin{proposition}\label{M-M-2}
Let $G$ be a finite non-cyclic abelian group. Suppose that $r\geq 2$. Then the following statements are equivalent for any $M\in\mathcal{M}(G)$.
\begin{enumerate}
\item[(1)] The cut set $\widetilde{M}$ of $\mathcal{P}(G)$ is minimal.
\item[(2)] $\overline{M}=\widetilde{M}$.
\item[(3)] Every Sylow subgroup of $G$ is non-cyclic.
\end{enumerate}
\end{proposition}

\begin{proof}
Clearly, $(1) \Rightarrow (2)$, since both $\overline{M}$ and $\widetilde{M}$ are cut-sets of $\mathcal{P}(G)$ with $\overline{M}\subseteq \widetilde{M}$. We have $(2)\Leftrightarrow (3)$ by Proposition \ref{M-M-1}. We prove $(3)\Rightarrow (1)$. So assume that every Sylow subgroup of $G$ is non-cyclic. Let $M\in\mathcal{M}(G)$. We need to show that $\mathcal{P}(G\setminus (\widetilde{M}-\{\alpha\}))$ is connected for every $\alpha\in \widetilde{M}$. Set $A=G\setminus M$ and $B=M\setminus \widetilde{M}$. By Proposition \ref{all-group}, $(A,B)$ is a separation of $\mathcal{P}(G\setminus\widetilde{M})$.  Clearly, $\mathcal{P}(B)$ is a clique and $\alpha$ is adjacent with every element of $B$. By Proposition \ref{nongen-adjacency}, there exists $\beta\in A=G\setminus M$ such that $\alpha\in\langle\beta\rangle$, that is, $\beta\sim\alpha$.
Since $r\geq 2$, by Proposition \ref{all-group1}, $\mathcal{P}(A)$ is connected. It follows that $\mathcal{P}(G\setminus (\widetilde{M}-\{\alpha\}))$ is connected.
\end{proof}

\section{Proof of Theorem \ref{main-1}}\label{sec-main-1}

We need the following elementary result in the proof of Theorem \ref{main-1} when $p_k\geq r+1$.

\begin{lemma}\label{inequality1}
If $q_1,q_2,\ldots, q_{t}$ are primes with $q_{1}< q_{2}< \cdots <q_{t}$, then $(t+1)\phi(q_{1}q_{2} \ldots q_{t})\geq q_{1} q_{2} \ldots q_{t}$. Further, equality holds if and only if $(t, q_1)=(1,2)$ or $(t, q_1,q_2)=(2,2,3)$.
\end{lemma}

\begin{proof}
Since $q_1\geq 2$, the inequality follows from the following:
\begin{equation*}
\begin{aligned}
\frac{\phi(q_1q_2\ldots q_{t})}{q_1q_2\ldots q_{t}} & = \left(1-\frac{1}{q_1}\right)\left(1-\frac{1}{q_2}\right)\cdots \left(1-\frac{1}{q_{t}}\right)\\
& \geq \left(1-\frac{1}{2}\right)\left(1-\frac{1}{3}\right)\cdots \left(1-\frac{1}{t+1}\right)= \frac{1}{t+1}.
\end{aligned}
\end{equation*}
Clearly, the above inequality is strict except in the cases $(t, q_1)=(1,2)$ and $(t, q_1,q_2)=(2,2,3)$.
\end{proof}
\medskip

\begin{proof}[{\bf Proof of Theorem \ref{main-1}}]
By Proposition \ref{nilpotent-mincut}, $Q$ is a minimal cut-set of $\mathcal{P}(G)$. Let $X$ be any minimal cut-set of $\mathcal{P}(G)$ different from $Q$. Then $Q\nsubseteq X$. Lemma \ref{max-cyclic} implies that every member of $\mathcal{M}(G)$ contains $P_i$, $i\neq k$ and so contains $Q$. Taking $D=Q$ in Proposition \ref{existence-0}, it follows that there exists $M\in\mathcal{M}(G)$ such that $\mathcal{P}(M\setminus X)=\mathcal{P}(M\setminus (X\cap M))$ is disconnected. So $\kappa(\mathcal{P}(M))\leq |X\cap M|$.
We have $|Q|=p_{1}^{n_{1}}\ldots p_{k-1}^{n_{k-1}}p_{k+1}^{n_{k+1}}\ldots p_{r}^{n_{r}}$ and $|M|= p_{1}^{n_{1}}\ldots p_{k-1}^{n_{k-1}}p_{k}^{l}p_{k+1}^{n_{k+1}}\ldots p_{r}^{n_{r}}$ for some integer $l$ with $1\leq l< n_{k}$. In order to prove the theorem, we show that $|X|> |Q|$ in both cases.

We first assume that $p_k\geq r+1$. Then $\phi(p_k)\geq r$. Since $1+\phi(|M|)\leq \kappa(\mathcal{P}(M))\leq |X\cap M|\leq |X|$, we have
\begin{eqnarray*}
|X| - |Q| & \geq & 1+\phi \left(p_{1}^{n_{1}}\ldots p_{k-1}^{n_{k-1}}p_{k}^{l}p_{k+1}^{n_{k+1}}\ldots p_{r}^{n_{r}}\right)-p_{1}^{n_{1}}\ldots p_{k-1}^{n_{k-1}}p_{k+1}^{n_{k+1}}\ldots p_{r}^{n_{r}}\\
& > & \phi\left(p_k^l\right) \phi \left(p_{1}^{n_{1}}\ldots p_{k-1}^{n_{k-1}}p_{k+1}^{n_{k+1}}\ldots p_{r}^{n_{r}}\right)-p_{1}^{n_{1}}\ldots p_{k-1}^{n_{k-1}}p_{k+1}^{n_{k+1}}\ldots p_{r}^{n_{r}}\\
&\geq &\phi(p_k)\phi(p_1\ldots p_{k-1}p_{k+1}\ldots p_r)-p_1\ldots p_{k-1}p_{k+1}\ldots p_r\\
&\geq &r\phi(p_1\ldots p_{k-1}p_{k+1}\ldots p_r)-p_1\ldots p_{k-1}p_{k+1}\ldots p_r\\
&\geq & 0.
\end{eqnarray*}
In the above, the last inequality holds by Lemma \ref{inequality1}. Thus $|X|> |Q|$.\\

We now assume that $2\phi(p_1\ldots p_{r-1}) > p_1\ldots p_{r-1}$. Then, by Theorem \ref{cyclic-connectivity}(i), we have
\begin{eqnarray*}
\kappa(\mathcal{P}(M))& = & \phi(|M|) + p_1^{n_1 -1}\ldots p_{k-1}^{n_{k-1} -1}p_{k}^{l-1}p_{k+1}^{n_{k+1} -1}\ldots p_r^{n_r -1} \left[p_1p_2\ldots p_{r-1} - \phi(p_1p_2\ldots p_{r-1})\right]\\
 & \geq & \phi(|M|) + p_1^{n_1 -1}\ldots p_{k-1}^{n_{k-1} -1}p_{k+1}^{n_{k+1} -1}\ldots p_r^{n_r -1} \left[p_1p_2\ldots p_{r-1} - \phi(p_1p_2\ldots p_{r-1})\right].
\end{eqnarray*}
Since $\kappa(\mathcal{P}(M))\leq |X\cap M|\leq |X|$ and $\phi(|M|)\geq \phi(p_1p_2\ldots p_k\ldots p_r)$, we get
\begin{eqnarray*}
|X| - |Q| & \geq & \kappa(\mathcal{P}(M))-|Q|\\
 & \geq &\phi(p_1p_2\ldots p_r)+p_1p_2\ldots p_{r-1}-\phi(p_1p_2\ldots p_{r-1})-p_1\ldots p_{k-1}p_{k+1}\ldots  p_r\\
 & = & \phi(p_1p_2\ldots p_{r-1})(p_{r}-2)-p_1\ldots p_{k-1}p_{k+1}\ldots p_{r-1}(p_{r}-p_k).
\end{eqnarray*}
If $k=r,$ then $|X|-|Q| \geq \phi(p_1p_2\ldots p_{r-1})(p_{r}-2)>0$ (note that $r\geq 2$ implies $p_r\geq 3$). So let $k\in \{1,2,\ldots ,r-1\}$. Since $2\phi(p_1\ldots p_{r-1}) > p_1\ldots p_{r-1}$, we have
$$\frac{p_1\ldots p_{k-1}p_{k+1}\ldots p_{r-1}}{\phi(p_1\ldots p_{k-1}p_{k+1}\ldots p_{r-1})}< \frac{2\phi(p_{k})}{p_{k}}.$$
Since $2\leq p_k<p_r,$ we have $p_r-p_k\leq p_r-2$ and $2\phi(p_k)\leq p_k\phi(p_k)$. Therefore,
$$\frac{2\phi(p_{k})}{p_{k}}\leq \frac{\phi(p_{k})(p_{r}-2)}{p_{r}-p_{k}}.$$
Combining the above two inequalities, we get
$$\frac{p_1\ldots p_{k-1}p_{k+1}\ldots p_{r-1}}{\phi(p_1\ldots p_{k-1}p_{k+1}\ldots p_{r-1})}< \frac{\phi(p_{k})(p_{r}-2)}{p_{r}-p_{k}}.$$
This gives $\phi(p_1p_2\ldots p_{r-1})(p_{r}-2)-p_1\ldots p_{k-1}p_{k+1}\ldots p_{r-1}(p_{r}-p_k)>0$ and so $|X|> |Q|$.
\end{proof}

\section{Proof of Theorem \ref{main-2}}\label{sec-main-2}

We start with the following result.

\begin{lemma}\label{max-min-cut}
Let $H$ be a cyclic group of order $p_1^{n_1} p_2^{n_2}$ and $X$ be a cut-set of $\mathcal{P}(H)$. If $p_1\geq 3$, then $|X|> |\widetilde{H}|$.
\end{lemma}

\begin{proof}
We have $|\widetilde{H}|=p_{1}^{n_{1}}p_{2}^{n_2}-\phi\left(p_{1}^{n_{1}}p_{2}^{n_2}\right)$. Since $\mathcal{P}(H\setminus X)$ is disconnected, we have $|X|\geq \kappa(\mathcal{P}(H))=\phi \left(p_{1}^{n_{1}}p_{2}^{n_2}\right)+p_{1}^{n_{1}-1}p_{2}^{n_2-1}$ by Theorem \ref{cyclic-connectivity}(ii). Then
\begin{eqnarray*}
|X| - |\widetilde{H}| & \geq & \phi \left(p_{1}^{n_{1}}p_{2}^{n_2}\right)+p_{1}^{n_{1}-1}p_{2}^{n_2-1}-\left(p_{1}^{n_{1}}p_{2}^{n_2}-\phi\left(p_{1}^{n_{1}}p_{2}^{n_2}\right) \right)\\
& = & p_1^{n_1 -1}p_2^{n_2 -1}\left[ p_1p_2-2p_1-2p_2+3\right] \\
& > & 0
\end{eqnarray*}
Here the last inequality holds as $p_1\geq 3$. So $|X|> |\widetilde{H}|$.
\end{proof}

In the rest of this section, $G$ is a finite non-cyclic abelian group of order $p_1^{n_1} p_2^{n_2}$ and $P_i$ is the Sylow $p_i$-subgroup of $G$, $i=1,2$.

\subsection{Proof of Theorem \ref{main-2}(i)}

The proof of the following result is similar to that of Theorem \ref{main-1}.

\begin{proposition}\label{P-1-cyclic}
Suppose that $p_1=2$, $P_1$ is non-cyclic and $P_2$ is cyclic. Then $P_2$ is a minimum cut-set of $\mathcal{P}(G)$ and so $\kappa(\mathcal{P}(G))=p_2^{n_2}$.
\end{proposition}

\begin{proof}
By Proposition \ref{nilpotent-mincut}, $P_2$ is a minimal cut-set of $\mathcal{P}(G)$. Let $X$ be any minimal cut-set of $\mathcal{P}(G)$ different from $P_2$. Then $P_2\nsubseteq X$. Since $P_2$ is cyclic, Proposition \ref{existence-0} implies that there exists $M\in\mathcal{M}(G)$ such that $\mathcal{P}(M\setminus X)=\mathcal{P}(M\setminus (X\cap M))$ is disconnected. So $\kappa(\mathcal{P}(M))\leq |X\cap M|$.

Again, since $P_2$ is cyclic, Lemma \ref{max-cyclic} implies that $M$ contains $P_2$ and so $|M|=2^lp_2^{n_2}$ for some $l$ with $1\leq l <n_1$. By Theorem \ref{cyclic-connectivity}(ii),
$$\kappa(\mathcal{P}(M))= \phi\left(2^lp_2^{n_2}\right) + 2^{l -1} p_2^{n_2 -1}\geq \phi\left(p_2^{n_2}\right) + p_2^{n_2 -1}=p_2^{n_2}.$$
Then $|X|-|P_2|\geq |X\cap M| - p_2^{n_2}\geq \kappa(\mathcal{P}(M)) - p_2^{n_2}\geq p_2^{n_2} - p_2^{n_2}=0$, giving $|X|\geq |P_2|$. So $\kappa(\mathcal{P}(G))=|P_2|=p_2^{n_2}$.
\end{proof}

\begin{proof}[{\bf Proof of Theorem \ref{main-2}(i)}]
If $p_1\geq 3$ and $P_i$ is non-cyclic, then Theorem \ref{main-1} implies that $P_j$ is the only minimum cut-set of $\mathcal{P}(G)$ and so $\kappa(\mathcal{P}(G))=|P_j|=p_j^{n_j}$, where $\{i,j\}=\{1,2\}$.

Assume that $p_1=2$. We have $p_2\geq 3$. If $P_2$ is non-cyclic, then Theorem \ref{main-1} again implies that $P_1$ is the only minimum cut-set of $\mathcal{P}(G)$ and so $\kappa(\mathcal{P}(G))=|P_1|=2^{n_1}$. If $P_1$ is non-cyclic, then $\kappa(\mathcal{P}(G))=p_2^{n_2}$ by Proposition \ref{P-1-cyclic} with $P_2$ being a minimum cut-set.
\end{proof}

We note that if $p_1=2$, $P_1$ is non-cyclic and $P_2$ is cyclic, then there might be minimum cut-sets of $\mathcal{P}(G)$ which are different from $P_2$. The following example justifies this statement.

\begin{example}\label{example}
Let $G= \langle a\rangle \times \langle b\rangle\times \langle x\rangle$, where $a,b$ are of order $2$ and $x$ is of order $3$. Then $|G|=2^2 \times 3$. By Proposition \ref{P-1-cyclic}, $\kappa(\mathcal{P}(G))=3$ and the Sylow $3$-subgroup $\langle x\rangle$ of $G$ is a minimum cut-set of $\mathcal{P}(G)$. Consider the subset $X_1=\{1,ax,ax^2\}$ of $G$ of size $3$. Taking $A=\{a\}$ and $B=G\setminus (X_1\cup A)$, one can see that $(A,B)$ is a separation of $\mathcal{P}(G\setminus X_1)$ and so $X_1$ is a minimum cut-set of $\mathcal{P}(G)$ different from $\langle x\rangle$. In fact, $\langle x\rangle$, $X_1$, $X_2=\{1,bx,bx^2\}$ and $X_3=\{1,abx,abx^2\}$ are precisely the minimum cut-sets of $\mathcal{P}(G)$.
\end{example}

\subsection{Proof of Theorem \ref{main-2}(ii)}

\begin{proof}[{\bf Proof of Theorem \ref{main-2}(ii)}]
Let $X$ be a minimal cut-set of $\mathcal{P}(G)$. By Propositions \ref{nilpotent-mincut} and \ref{M-M-2}, each of $P_1$, $P_2$ and $\widetilde{C}$ is a minimal cut-set of $\mathcal{P}(G)$. We may assume that $X\neq P_1$ and $X\neq P_2$. Then $P_1\nsubseteq X$ and $P_2\nsubseteq X$. We shall show that $|X|\geq |\widetilde{C}|$.

If the subgraph $\mathcal{P}(M\setminus X)$ is disconnected for some $M\in\mathcal{M}(G)$, then $p_1\geq 3$ implies that $|X|> |\widetilde{M}|\geq |\widetilde{C}|$ by Lemma \ref{max-min-cut} and Proposition \ref{size-compare}.
So we may suppose that $\mathcal{P}(M\setminus X)$ is connected for every $M\in\mathcal{M}(G)$. Then Proposition \ref{gen-connected} implies that $X$ does not contain generators of any maximal cyclic subgroup of $G$.
Fix a separation $(A,B)$ of $\mathcal{P}(G\setminus X)$. By Proposition \ref{contain}, we may assume that both $P_1\setminus X$ and $P_2\setminus X$ are contained in $A$. We have $B\neq \emptyset$. Let $M\in \mathcal{M}(G)$ containing an element of $B$. Then all the generators of $M$ are in $B$, and so $X$ must contains both $M\cap P_{1}$ and $M\cap P_{2}$. Then
$|X|\geq |(M\cap P_{1})\cup (M\cap P_{2})|\geq |(C\cap P_{1})\cup (C\cap P_{2})|=p_1+p_2-1=|\widetilde{C}|$.
\end{proof}

\noindent As a consequence of Theorem \ref{main-2}(ii), we have the following.

\begin{corollary}
Suppose that both $P_1, P_2$ are non-cyclic and elementary abelian. If $p_1\geq 3$, then $\kappa(\mathcal{P}(G))=\min \left\{|P_1|,p_1+p_2 -1\right\}$.
\end{corollary}

\subsection{Proof of Theorem \ref{main-2}(iii)}

We first prove the following.

\begin{proposition}\label{2-gen-discon}
Suppose that both $P_1$ and $P_2$ are non-cyclic and that $P_1$ contains a maximal cyclic subgroup of order $2$ (so $p_{1}=2$). If $X$ is a minimum cut-set of $\mathcal{P}(G)$, then $X$ does not contain generators of any maximal cyclic subgroup of $G$.
\end{proposition}

\begin{proof}
By Proposition \ref{nilpotent-mincut}, $P_1$ is a minimal cut-set of $\mathcal{P}(G)$. Let $C\in\mathcal{M}(G)$ be of minimum possible order. By Proposition \ref{M-M-2}, $\widetilde{C}$ is a minimal cut-set of $\mathcal{P}(G)$.   Since $X$ is a minimum cut-set of $\mathcal{P}(G)$, we have $|X|\leq \min \left\{|P_1|,|\widetilde{C}|\right\}$. Since $P_1$ contains a maximal cyclic subgroup of order $2$, $|C|=2p_2^c$ for some positive integers $c$. Then $|\widetilde{C}|=|C|-\phi(|C|)=2p_2^c- \phi(2p_2^c)= p_{2}^{c}+p_{2}^{c-1}$.

On the contrary, suppose that $X$ contains a generator of some $M\in\mathcal{M}(G)$. Then $X$ contains all the generators of $M$ by Lemma \ref{equi-class}. We have $|M|=2^l p_{2}^{m}$ for some positive integers $l, m$ with $m\geq c$. Then $| P_{1}\setminus X |\geq 2^l$. Otherwise, $| P_{1} \cap X |> 2^{n_{1}}-2^l$ and so $|X|  \geq  | P_{1} \cap X |+\phi \left(2^l p_{2}^{m}\right) >  2^{n_{1}}-2^l+2^{l-1}\phi \left(p_{2}^{m}\right) \geq 2^{n_1}=|P_{1}|$, a contradiction as $p_2\geq 3$.\\

\noindent {\bf Claim-1}: $X$ does not contain any element of order $p_{2}^{k}$ with $k\geq c$.

Suppose that $x$ is such an element in $X$. Then $\langle x \rangle$ is of order $p_{2}^{k}$ and by Lemma \ref{equi-class}, all the generators of $\langle x \rangle$ are in $X$. Since $1\in X$ and $k,m\geq c$, we have
\begin{equation*}
|X| \geq 1+ \phi \left(p_{2}^{k}\right)+\phi \left(2^lp_{2}^{m}\right) \geq 1+2\phi(p_{2}^{c})=1+p_{2}^{c}+p_{2}^{c-1}(p_{2}-2).
\end{equation*}
Since $p_3\geq 3$, we get that $|X| \geq 1+p_{2}^{c}+p_{2}^{c-1} >  p_{2}^{c}+p_{2}^{c-1} =|\widetilde{C}|$, a contradiction.\\

\noindent {\bf Claim-2}: $X\setminus M$ does not contain any element of order $2^t p_{2}^{k}$  with $t\geq 1$ and $k\geq c$.

Suppose that $y$ is such an element in $X\setminus M$. Then $\langle y \rangle$ is of order $2^t p_{2}^{k}$ and all the generators of $\langle y \rangle$ are in $X$ by Lemma \ref{equi-class} but none of them is in $M$. Then, as $1\in X$ and $k,m\geq c$, we have
\begin{equation*}
|X| \geq  1+\phi \left(2^l p_{2}^{m}\right)+\phi \left(2^t p_{2}^{k}\right) \geq  1+2\phi (p_{2}^{c})>p_{2}^{c}+p_{2}^{c-1}=|\widetilde{C}|,
\end{equation*}
a contradiction. \\

Note that, by Claim-2, $X$ does not contain generators of any maximal cyclic subgroup of $G$ different from $M$.  Since $1\in X$, $|P_{1}\setminus X |\geq 2^l$ and $M\cap P_1$ contains $2^{l}-1$ non-identity elements, it follows that there exists $v\in P_{1}\setminus X$ such that $v\notin M$. We shall get a contradiction by showing that $\mathcal{P}(G\setminus X)$ is connected. It is enough to show that every element of $G\setminus X$ is connected to $v$ in $\mathcal{P}(G\setminus X)$. Consider an arbitrary element $\alpha\in G\setminus X$, $\alpha\neq v$.

First assume that $\alpha\notin M$. Let $M_1=\langle \beta \rangle$ be a maximal cyclic subgroup of $G$ containing $\alpha$, where $o(\beta)=2^{l_1}p_2^{m_1}=|M_1|$ for some positive integers $l_1,m_1$ with $m_1\geq c$. Let $u\in M_1$ be such that $o(u)=p_{2}^{m_{1}}$. Since $m_1\geq c$, we have $u\notin X$ by Claim-1 and $\beta\notin X$ by Claim-2. The element $vu$ is of order $2^s p_{2}^{m_{1}}$, where $o(v)=2^s$. If $vu\in X$, then $vu\in M$ by Claim-2 and it follows that $v\in M$, a contradiction. So $vu\notin X$. Then $\alpha \sim u \sim vu \sim v$ or $\alpha \sim \beta \sim u \sim vu \sim v$ is a walk between $\alpha$ and $v$ in $\mathcal{P}(G\setminus X)$ according as $\alpha=\beta$ or not.

Now assume that $\alpha \in M\setminus X$. Then $\alpha$ is a non-generator of $M$, as $X$ contains all the generators of $M$. By Corollary \ref{coro-non-adj}, let $M_2=\langle \delta \rangle$ be a maximal cyclic subgroup of $G$ different from $M$ and containing $\alpha$, where $o(\delta)=|M_2|=2^{l_2}p_2^{m_2}$ for some positive integers $l_2,m_2$ with $m_2\geq c$. Then $\delta\notin M$ and $\alpha \sim \delta$. Let $u'\in M_2$ be such that $o(u')=p_{2}^{m_{2}}$. By a similar argument as in the previous paragraph, we have $u'\notin X$, $\delta\notin X$ and $vu'\notin X$. Then $\alpha \sim \delta \sim u^{\prime} \sim vu^{\prime} \sim v$ is a walk between $\alpha$ and $v$ in $\mathcal{P}(G\setminus X)$. This completes the proof.
\end{proof}

\begin{proof}[{\bf Proof of Theorem \ref{main-2}(iii)}]
Let $X$ be a minimum cut-set of $\mathcal{P}(G)$. By Propositions \ref{nilpotent-mincut} and \ref{M-M-2}, each of $P_1$, $P_2$ and $\widetilde{C}$ is a minimal cut-set of $\mathcal{P}(G)$. It is enough to prove that $|X|\geq \min \left\{|P_1|,|P_2|,|\widetilde{C}|\right\}$. Without loss of generality, we may assume that $X\neq P_1$ and $X\neq P_2$. Then $P_1\nsubseteq X$ and $P_2\nsubseteq X$. We shall show that $|X|\geq |\widetilde{C}|$.

If $p_{1}\geq 3$ and the subgraph $\mathcal{P}(M\setminus X)$ is disconnected for some $M\in\mathcal{M}(G)$, then $|X|> |\widetilde{M}|\geq |\widetilde{C}|$ by Lemma \ref{max-min-cut} and Proposition \ref{size-compare}.
So, for $p_{1}\geq 3$, we may suppose that $\mathcal{P}(M\setminus X)$ is connected for every $M\in\mathcal{M}(G)$. Then, for all $p_1\geq 2$, Propositions \ref{gen-connected} and \ref{2-gen-discon} imply that $X$ does not contain generators of any maximal cyclic subgroup of $G$.

Fix a separation $(A,B)$ of $\mathcal{P}(G\setminus X)$. By Proposition \ref{contain}, we may assume that both $P_1\setminus X$ and $P_2\setminus X$ are contained in $A$. Since $B\neq \emptyset$, let $M_1\in \mathcal{M}(G)$ containing an element of $B$. Then all the generators of $M_1$ are in $B$ and $X$ contains both $M_{1}\cap P_{1}$ and $M_{1}\cap P_{2}$.

Now consider an element $\beta \in \widetilde{M_{1}}\setminus(P_1\cup P_2)$. Clearly, $\beta\notin A$. First suppose that $\beta \in B$. Since $\beta$ is a non-generator of $M_{1}$, by Corollary \ref{coro-non-adj}, there exists $M_2\in \mathcal{M}(G)$ with $M_2\neq M_1$ such that $\beta \in M_{2}$. Then all the generators of $M_{2}$ are also in $B$, and $X$ contains both $M_{2}\cap P_{1}$ and $M_{2}\cap P_{2}$. Since $P_1$ is elementary abelian, we have $|C|=p_1p_2^c$, $|M_1|=p_1p_2^{m_1}$ and $|M_2|=p_1p_2^{m_2}$, where $c,m_1,m_2$ are positive integers with $m_1,m_2\geq c$. Since $M_1\cap P_1=\langle \beta\rangle\cap P_1=M_2\cap P_1$ (each of order $p_1$) and $M_1\neq M_2$, the elements of $M_{2}$ of order $p_{2}^{m_{2}}$ are not in $M_{1}$. Then
$$|(M_{2}\cap P_{2})\setminus(M_{1}\cap P_{2})| \geq \phi(p_{2}^{m_{2}})\geq \phi(p_{2}^{c})$$
and
$$|(M_{1}\cap P_{1})\cup(M_{1}\cap P_{2})|=p_{1}+p_{2}^{m_{1}}-1 \geq p_{1}+p_{2}^{c}-1.$$
Using these two inequalities and the facts that $p_2>p_1$ and $M_{i}\cap P_{1},\;M_{i}\cap P_{2}$ are contained in $X$ for $i=1,2$, we get
$|X|  \geq  \phi (p_{2}^{c})+p_{1}+p_{2}^{c}-1 >  p_{2}^{c-1}[p_1+p_2-1] =|\widetilde{C}|$.

Now suppose that every element of $\widetilde{M_{1}}\setminus (P_1\cup P_2)$ is in $X$. Since $M_{1}\cap P_{1}$ and $M_{1}\cap P_{2}$ are contained in $X$, it follows that $\widetilde{M_{1}}\subseteq X$. Then $|X|\geq |\widetilde{M_{1}}|\geq |\widetilde{C}|$ by Proposition \ref{size-compare}. This completes the proof.
\end{proof}

\section{Proof of Theorem \ref{main-3}}\label{sec-main-3}

Let $G$ be a finite non-cyclic abelian group of order $p_1^{n_1}p_2^{n_2}p_3^{n_3}$. If $p_1\geq 3$, then $2\phi(p_1p_2) > p_1p_2$ and so Theorem \ref{main-3}(iii) follows from Theorem \ref{main-1}. In the rest of this section, we assume that $p_1=2$.

\begin{proof}[{\bf Proof of Theorem \ref{main-3}(ii)}]
Suppose that $P_1$ is cyclic and that either $P_2$ or $P_3$ is non-cyclic. If $P_3$ is non-cyclic, then the result follows from Theorem \ref{main-1} as $p_3\geq 5>r+1$ with $r=3$.

Suppose that $P_2$ is non-cyclic. By Proposition \ref{nilpotent-mincut}, $P_1P_3$ is a minimal cut-set of $\mathcal{P}(G)$. Let $X$ be any minimal cut-set of $\mathcal{P}(G)$ different from $P_1P_3$. Then $P_1P_3\nsubseteq X$. Taking $D=P_1P_3$ in Proposition \ref{existence-0}, we get $M\in\mathcal{M}(G)$ such that $\mathcal{P}(M\setminus X)=\mathcal{P}(M\setminus (X\cap M))$ is disconnected. So $\kappa(\mathcal{P}(M))\leq |X\cap M|$.
It is enough to show that $|X|> |P_1P_3|$.

We have $|P_1P_3|=2^{n_{1}}p_{3}^{n_{3}}$ and $|M|=2^{n_1}p_{2}^{l}p_{3}^{n_{3}}$ for some positive integer $l<n_{2}$.
By Theorem \ref{cyclic-connectivity}(iii)(a),
\begin{eqnarray*}
\kappa(\mathcal{P}(M)) & = & \phi\left(2^{n_1}p_2^l p_3^{n_3}\right) + 2^{n_1 -1} p_2^{l -1}\left[(p_2-1)p_3^{n_3-1}+2\right]\\
& \geq & \phi(2^{n_1}p_2p_3^{n_3}) +2^{n_1 -1}p_3^{n_3-1}\phi(p_2) +2^{n_1}.
\end{eqnarray*}
So,
\begin{eqnarray*}
\kappa(\mathcal{P}(M)) - |P_1P_3| & \geq & 2^{n_1 -1}p_3^{n_3-1}\left[\phi(p_2p_3)+\phi(p_2)-2p_3\right]+2^{n_1}\\
  &  = & 2^{n_1 -1}p_3^{n_3-1}\left[(\phi(p_2)-2)p_3\right]+2^{n_1} >0
\end{eqnarray*}
Therefore, $|X|\geq |X\cap M| \geq \kappa(\mathcal{P}(M))>|P_1P_3|$. This completes the proof.
\end{proof}

We next prove Theorem \ref{main-3}(i). Suppose that $P_1$ is non-cyclic and that $P_2$ and $P_3$ are cyclic. Let $M \in \mathcal{M}(G)$. Then $|M|=2^{m}p_{2}^{n_{2}}p_{3}^{n_{3}}$ for some positive integer $m<n_{1}$. For every divisor $d$ of $|M|$, we define the following two sets:
\begin{enumerate}
\item[] $E(M,d)$ = the set of all elements of $M$ whose order is $d$,
\item[] $S(M,d)$ = the set of all elements of $M$ whose order divides $d$.
\end{enumerate}
Since $M$ is cyclic, $S(M,d)$ is the unique cyclic subgroup of $M$ of order $d$ and $E(M,d)$ is precisely the set of all generators of $S(M,d)$. Let $d_j=2^{m}p_{2}^{n_{2}}p_{3}^{j}$ for $1\leq j\leq n_3$ and define the subset $\Gamma(M)$ of $M$ as:
$$\Gamma(M)=\left(\underset{j=1}{\overset{n_{3}}{\bigcup}} E\left(M, d_j\right)\right)\bigcup S\left(M, 2^{m}p_{2}^{n_{2}-1}\right)\bigcup S\left(M,2^{m-1}p_{2}^{n_{2}}\right).$$
It can be calculated that $|\Gamma(M)|=\phi\left(|M|\right)+2^{m-1}p_{2}^{n_{2}-1}\left[(p_2 -1)p_{3}^{n_{3}-1}+2\right]$.

Set $A=E\left(M, 2^{m}p_{2}^{n_{2}}\right)$ and $B=M \setminus (A\cup \Gamma(M))$. Clearly, both $A$ and $B$ are non-empty and $A\cap B=\emptyset$. Any element of $B$ has order of the form $2^lp_2^kp_3^t,$ where $0\leq l\leq m$, $0\leq k\leq n_2$, $1\leq t \leq n_3$ and $(l,k)\neq (m,n_2)$. Let $x$ be an element of $M$ such that $o(x)=p_3$. Then $x\in B$ and every other element of $B$ is adjacent to $x$ (as $t\geq 1$). So $\mathcal{P}(B)$ is connected. Since $A$ is the set of all generators of the cyclic subgroup $S\left(M,2^{m}p_{2}^{n_{2}}\right)$, $\mathcal{P}(A)$ is a clique and so is connected. Since $(l,k)\neq (m,n_2)$ and $t\geq 1$, observe that none of the elements of $A$ can be obtained as an integral power any element of $B$ and vice-versa. Thus $(A,B)$ is a separation of $\mathcal{P}(M\setminus\Gamma(M))$ and hence $\Gamma(M)$ is a cut-set of $\mathcal{P}(M)$ with two connected components $A$ and $B$. We have
$$\kappa(\mathcal{P}(M))=\phi\left(|M|\right)+2^{m-1}p_{2}^{n_{2}-1}\left[(p_2 -1)p_{3}^{n_{3}-1}+2\right]=|\Gamma(M)|$$
by Theorem \ref{cyclic-connectivity}(iii)(a). In fact, $\Gamma(M)$ is the only minimum cut-set of $\mathcal{P}(M)$, see the second last statement of Theorem \ref{cyclic-connectivity}.

\begin{lemma} \label{max-connectivity}
For $M\in\mathcal{M}(G)$, $\Gamma(M)$ is a minimal cut-set of $\mathcal{P}(G)$ with two connected components $A=E\left(M, 2^{m}p_{2}^{n_{2}}\right)$ and $B_1=G \setminus(A\cup\Gamma(M))$. In particular, $\kappa(\mathcal{P}(G))\leq |\Gamma(M)|=\kappa(\mathcal{P}(M))$.
\end{lemma}

\begin{proof}
Note that $B_1=B\cup (G\setminus M)$, where $B$ is defined as in the above discussion. We show that $(A,B_1)$ is a separation of $\mathcal{P}(G\setminus \Gamma(M))$. Clearly, both $A$ and $B_1$ are non-empty and $A\cap B_1=\emptyset$. Suppose that $x\sim y$ for some $x\in A$ and $y\in B_1$. Then $x\in M$ with $o (x)=2^{m}p_{2}^{n_{2}}$, and $y\in B_1 \setminus B$ as $(A,B)$ is a separation of $\mathcal{P}(M\setminus\Gamma(M))$. So $y\in N$ for some $N\in\mathcal{M}(G)$ with $M\neq N$. We have $x\in \langle y \rangle$ or $y\in \langle x\rangle$. If $y\in \langle x\rangle$, then $y\in M$, a contradiction to that $y\in B_1\setminus B=G\setminus M$. Suppose that $x\in \langle y \rangle$. Then $x\in N$. The facts that $x\in M$ and $2^m\mid o (x)$ imply that $\langle x\rangle\cap P_1$ is precisely the maximal cyclic subgroup $M\cap P_1$ of $P_1$. Since $x\in N$, it follows that the maximal cyclic subgroup $N\cap P_1$ of $P_1$ must be equal to $\langle x\rangle\cap P_1$. Thus $M\cap P_1=N\cap P_1$. Since $M\cap N$ contains $P_2 P_3$, we get that $M=N$, a contradiction. Thus $\Gamma(M)$ is a cut-set of $\mathcal{P}(G)$.

We know that both $\mathcal{P}(A)$ and $\mathcal{P}(B)$ are connected. We claim that $\mathcal{P}(B_1)$ is also connected. Let $a\in B_1\setminus B$. Then $a\notin M$. So there exists $N \in \mathcal{M}(G)$ with $N\neq M$ such that $a\in N$. Let $\alpha=\alpha_1\alpha_2\alpha_3$ be a generator of $N$, where $\alpha_1\in N\cap P_1$, $\alpha_2\in P_2$ and $\alpha_3\in P_3$. Then $\alpha\in B_1$ as $\alpha\notin M$. Since $o(\alpha_2\alpha_3)= p_{2}^{n_{2}}p_{3}^{n_{3}}$, we have $\alpha_2\alpha_3\in B$.
If $a=\alpha$, then $a=\alpha\sim\alpha_2\alpha_3$ in $\mathcal{P}(B_1)$. If $a\neq \alpha$, then $a\sim\alpha\sim \alpha_2\alpha_3$ is a path in $\mathcal{P}(B_1)$. Thus every element of $B_1\setminus B$ is connected to an element of $B$ in $\mathcal{P}(B_1)$. Since $\mathcal{P}(B)$ is connected, it follows that $\mathcal{P}(B_1)$ is connected. Now the minimality of the cut-set $\Gamma(M)$ of $\mathcal{P}(G)$ follows from the fact that $\Gamma(M)$ is a minimum cut-set of $\mathcal{P}(M)$.
\end{proof}

\begin{proof}[{\bf Proof of Theorem \ref{main-3}(i)}]
Let $X$ be a minimum cut set of $\mathcal{P}(G)$. Proposition \ref{nilpotent-mincut} and Lemma \ref{max-connectivity} imply that $|X|\leq \min \left\{|P_2P_{3}|,\kappa(\mathcal{P}(C))\right\}$, where $C$ is a maximal cyclic subgroup of $G$ of minimum possible order. If $P_2P_{3}\subseteq X$, then $|X|\geq |P_2P_{3}|$. Assume that $P_2P_{3}\nsubseteq X$. Then, by taking $D=P_2P_3$ in Proposition \ref{existence-0}, we get $M\in\mathcal{M}(G)$ such that $\mathcal{P}(M\setminus X)=\mathcal{P}(M\setminus (X\cap M))$ is disconnected. We have $\kappa(\mathcal{P}(M)) \geq \kappa(\mathcal{P}(C))$, this can be observed from Theorem \ref{cyclic-connectivity}(iii)(a) as $|M|$ and $|C|$ may differ only in the power of $2$. So
$$|X| \geq |X\cap M|\geq \kappa(\mathcal{P}(M)) \geq \kappa(\mathcal{P}(C)).$$
Therefore, $\kappa(\mathcal{P}(G))=|X|= \min \left\{|P_2P_{3}|,\kappa(\mathcal{P}(C))\right\}$. Now let $|C|=2^{c}p_2^{n_2}p_3^{n_3}$ for some positive integer $c$. Using Theorem \ref{cyclic-connectivity}(iii)(a) again, we calculate that
$$\kappa(\mathcal{P}(C))-|P_2P_3|=p_{2}^{n_2-1}p_{3}^{n_{3}-1}\left[\left(2^{c-1}\phi(p_2)-p_2\right)p_3 \right]+2^c p_2^{n_2 -1}.$$
If $c>1$, then $2^{c-1}\phi (p_{2})> p_{2}$ as $p_2\geq 3$. It follows that $\kappa(\mathcal{P}(C))>|P_2P_3|$ and so $\kappa(\mathcal{P}(G))=\min \left\{|P_2P_{3}|,\kappa(\mathcal{P}(C))\right\}=|P_2P_3|$.
If $c=1$, then
\begin{eqnarray*}
\kappa(\mathcal{P}(C))-|P_2P_3|& = &p_{2}^{n_2-1}p_{3}^{n_{3}-1}\left[(\phi(p_2)-p_2)p_3 \right]+2 p_2^{n_2 -1}\\
  & = & p_{2}^{n_2-1}\left[2 - p_3^{n_3}\right] < 0.
\end{eqnarray*}
So $\kappa(\mathcal{P}(C))<|P_2P_3|$ and hence $\kappa(\mathcal{P}(G))=\min \left\{|P_2P_{3}|,\kappa(\mathcal{P}(C))\right\}=\kappa(\mathcal{P}(C))$. This completes the proof.
\end{proof}

\vskip .5cm

\noindent{\bf Addresses}:\\

\noindent Sriparna Chattopadhyay, Kamal Lochan Patra, Binod Kumar Sahoo\\

\noindent 1) School of Mathematical Sciences,\\
National Institute of Science Education and Research (NISER), Bhubaneswar,\\
P.O.- Jatni, District- Khurda, Odisha - 752050, India\medskip

\noindent 2) Homi Bhabha National Institute (HBNI),\\
Training School Complex, Anushakti Nagar,\\
Mumbai - 400094, India,\medskip

\noindent {\bf E-mails}: sriparna@niser.ac.in, klpatra@niser.ac.in, bksahoo@niser.ac.in\\
\end{document}